\documentclass[11pt,reqno]{amsart}
\usepackage{mathrsfs}
\usepackage{amsfonts}

\usepackage{amssymb,amsmath,amsfonts,amsthm}
\usepackage{bm,cite,graphicx,color}
\usepackage{epstopdf}
\usepackage{subfigure}
\usepackage{url}
\usepackage[toc,page]{appendix}

\newtheorem{theorem}{Theorem}[section]
\newtheorem{lemma}[theorem]{Lemma}

\newtheorem{remark}[theorem]{Remark}

\numberwithin{equation}{section}

\allowdisplaybreaks[4]

\begin{document}

\title[An inverse random source problem]{Numerical solution of an inverse random source problem for the time fractional diffusion equation via PhaseLift}

\author{Yuxuan Gong}
\address{School of Mathematical Sciences, Zhejiang University, Hangzhou, Zhejiang 310027, China.}
\email{yxgong@zju.edu.cn}

\author{Peijun Li}
\address{Department of Mathematics, Purdue University, West Lafayette, Indiana 47907, USA.}
\email{lipeijun@math.purdue.edu}

\author{Xu Wang}
\address{Department of Mathematics, Purdue University, West Lafayette, Indiana 47907, USA.}
\email{wang4191@purdue.edu}

\author{Xiang Xu}
\address{School of Mathematical Sciences, Zhejiang University, Hangzhou, Zhejiang 310027, China.}
\email{xxu@zju.edu.cn}

\thanks{The research of P. Li is supported in part by the NSF grant DMS-1912704. The research of X. Xu was partially supported by NSFC 11621101, 12071430 and the Fundamental Research Funds for the Central Universities.}



\begin{abstract}
This paper is concerned with the inverse random source problem for a stochastic time fractional diffusion equation, where the source is assumed to be driven by a Gaussian random field. The direct problem is shown to be well-posed by examining the well-posedness and regularity of the solution for the equivalent stochastic two-point boundary value problem in the frequency domain. For the inverse problem, the Fourier modulus of the diffusion coefficient of the random source is proved to be uniquely determined by the variance of the Fourier transform of the boundary data. As a phase retrieval for the inverse problem, the PhaseLift method with random masks is applied to recover the diffusion coefficient from its Fourier modulus. Numerical experiments are reported to demonstrate the effectiveness of the proposed method. 
\end{abstract}

\maketitle

\section{Introduction}

In the past two decades, the differential equations involving fractional-order derivatives, known as the fractional differential equations (FDEs), have received increasing attention in applied disciplines. Such models are able to capture more faithfully the dynamics of anomalous diffusion processes in amorphous materials. Consequently, fundamentally different physics can be obtained. For instance, the FDEs can be used to model some anomalous diffusions in a highly heterogeneous aquifer \cite{Adams1992Field}, underground environmental problems \cite{Hatano1998Dispersive}, the relaxation phenomena in complex viscoelastic materials \cite{Giona1992Fractional}, and non-Markovian diffusion processes with memory \cite{Metzler2000Random}. 

Motivated by the significant applications in scientific and industrial fields, the research of inverse problems has gone tremendous developments in the past decade. Recently, the inverse problems on FDEs have become an active research field. In particular, the inverse problems for the time fractional diffusion equations have been extensively studied mathematically and numerically. Generally, the inverse source problems of FDEs are to determine the time-dependent or the space-dependent source functions by using the space-dependent or time-dependent data \cite{Aleroev2013Determination,Aziz2016Identification,Ismailov2016Inverse,Xu2011Inverse}. There are also cases where boundary conditions are used as the data\cite{Murio2008Source,Wei2016Uniqueness}. Although there has been a lot of research done for the inverse problems of FDEs, there is little work on the stochastic inverse problems for the FDEs. 

The inverse random source problems refer to the inverse source problems that involve uncertainties. Due to the randomness and uncertainty, compared to the deterministic counterparts, stochastic inverse problems have more difficulties in addition to the existing obstacle of ill-posedness. For the inverse random source scattering problem, where the wave propagation is governed by the stochastic Helmholtz equation driven by a white noise, it is shown in \cite{Devaney1979Inverse} that the correlation of the random source could be determined uniquely by the correlation of the random wave field. Effective computational methods are developed in \cite{Bao2010Numerical,Bao2014Inverse,Bao2016Inverse,Bao2017Inverse,Li2011Inverse,Li2017Inverse,Li2017Stability} for the white noise model, where statistical properties, such as the mean and the variance, of the random source are recovered by using the boundary measurements of the random wave field at multiple frequencies.

There is much less work on the inverse random source problems for the FDEs. Consider the stochastic time fractional diffusion equation 
\begin{equation*}
\begin{cases}
\partial^{\alpha}_t u (x,t)-\Delta u (x,t)=f (x)h (t)+g (x)\dot{B}^H (t),\quad& (x,t)\in D\times (0,T),\\
u (x,t)=0,\quad& (x,t)\in\partial D\times[0,T],\\
u (x,0)=0, \quad&x\in\overline{D}
\end{cases}
\end{equation*}
where $\partial_t^\alpha$ denotes the Caputo fractional derivative with $\alpha\in (0,1)$, $B^H$ is the fractional temporal Brownian motion with Hurst parameter $H\in(0,1)$ and $\dot B^H$ denotes the formal derivative of $B^H$ with respect to the time $t$, $D$ is a bounded domain with the Lipschitz boundary $\partial D$, and $f, g$ are two deterministic functions supported in $D$. Given $f$ and $g$, the direct source problem is to determine $u$; while the inverse source problem is to determine the unknowns $f$ and $g$ that generate a prescribed $u$. For this kind of random sources perturbed by a time-dependent random noise, the mild solution of the problem can be obtained by using the Mittag--Leffler functions, and it may lead to the reconstruction formulas between the unknowns and the measurements. In \cite{Niu2020Inverse}, $f$ and $|g|$ are proved to be uniquely determined by the mean and covariance of the final data $u (x,T)$, respectively, under the conditions that $\alpha\in (\frac{1}{2},1)$ and $H=\frac12$, i.e., the white noise case. It is also pointed out that the inverse problem is not stable as a small variance of the data may lead to a huge error on the reconstructions. The inverse problem for a generalized case with $\alpha\in(0,1)$ and $H\in(0,1)$ is considered in \cite{Feng2020Inverse}, where $f$ and $|g|$ are recovered from the same statistics of the final data $u(x,T)$. Similar results are obtained about the uniqueness and the stability of the inverse problem. If the random sources are perturbed by a space-dependent noise, the mild solution approach is not available anymore since the spatial noise may not be regular enough to guarantee the well-posedness of the problem. New techniques need to be explored to reconstruct the diffusion coefficient of this kind of random sources. 

In this paper, we consider the one-dimensional stochastic time fractional diffusion equation
\begin{equation}\label{question}
\begin{cases}
  \partial^{\alpha}_t u (x,t)-\partial_{xx}u (x,t)=F (t)\dot{W_x},\quad& (x,t)\in (0, 1)\times\mathbb R_+,\\
  u (x,0)=0,\quad& x\in [0, 1],\\
  \partial_xu (0,t)=0, \quad u (1,t)=0,\quad &t\in\mathbb R_+,
\end{cases}
\end{equation}
where $W_x$ is the spatial Brownian motion, $\dot W_x$ denotes the formal derivative with respect to $x$ and is known as the  white noise, $F(t)$ is a deterministic function and denotes the diffusion coefficient of the random source, and the Caputo fractional derivative $\partial^\alpha_t u$, $\alpha\in (0,1)$, is defined by 
\begin{equation*}
\partial^{\alpha}_t u (x,t)=\frac{1}{\Gamma (1-\alpha)}\int_0^t\partial_s u (x,s) (t-s)^{-\alpha}ds
\end{equation*}
with $\Gamma (\alpha)=\int_0^\infty e^{-s}s^{\alpha-1}ds$ being the Gamma function.

To deal with the initial boundary value problem of the stochastic diffusion equation \eqref{question}, we consider an equivalent 
stochastic two-point boundary value problem in the frequency domain, and study the well-posedness and regularity of the solution  based on the estimate of the Green function. It then leads to the well-posedness of \eqref{question} by showing that the solution to \eqref{question} is the inverse Fourier transform of the solution to the equivalent problem in the frequency domain.
For the inverse problem, the Fourier modulus $|\hat F|$ of the diffusion coefficient $F$ is proved to be uniquely determined by the variance of the solution to the equivalent problem in the frequency domain. However, the recovery of $F$ from $|\hat F|$ is apparently not unique, since we measure $|\hat F|$ instead of $\hat F$ and lose information about the phase of $F$. If we could retrieve the phase of $F$, then it would be trivial to recover $F$. This kind of problem, i.e., the reconstruction of a signal from the magnitude of its Fourier transform, is generally known as the phase retrieval \cite{Patterson1934AFourier,Patterson1944Ambiguities}. It arises in many applications such as diffraction imaging, optics and quantum mechanics, and is usually ill-posed and notoriously difficult to solve. A large amount of methods have been proposed to solve the phase retrieval problem \cite{Kishore2015Phase}. These approaches can be broadly classified into two categories: utilizing either a priori information about the 
signal $F$ or additional measurements of the modulus $|\hat F|$. In this work, we adopt the latter and apply the PhaseLift method with random masks to collect additional measurements of the modulus $|\hat F|$, which may be used to uniquely determine $|F|$, which is the modulus of the diffusion coefficient of the random source. Numerical experiments are reported to demonstrate the effectiveness of the proposed method. 

The rest of the paper is organized as follows. In Section \ref{preliminaries}, the time fractional derivative and its properties are introduced. Section \ref{Direct_Problem} is concerned with the well-posedness of the direct problem. The inverse problem is addressed in Section \ref{Inverse_Porblem}, which includes the recovery of $|\hat F|$, the phase retrieval problem, and its numerical method. Section \ref{Numerical_Exps} presents two numerical examples to illustrate the effectiveness of the proposed method. The paper is concluded with some general remarks and directions for the future research in Section \ref{Conclusion}.

\section{The time fractional derivative}\label{preliminaries}

In this section, we discuss the Fourier transform of the Caputo fractional derivative, which serves as a basis to convert the initial boundary value problem of the stochastic time fractional differential equation \eqref{question} into an equivalent stochastic two-point boundary value problem in the frequency domain. 

\begin{lemma}\label{lm:frac}
Let $v$ be a causal function, i.e., $v (t)=0$ if $t\le0$, whose fractional derivative $\partial_t^\alpha v$ is well-defined in $L^2 (\mathbb R)$. Then the Fourier transform of the fractional derivatives of $v$ satisfies 
\[
\mathcal {F}[\partial^\alpha_t v] (\omega)= ({\rm i}\omega)^\alpha\hat{v} (\omega)\quad\forall\, \alpha\in (0,1],
\]
where 
\[
\hat v (\omega)=\mathcal{F}[v] (\omega):=\int_{\mathbb R}e^{-{\rm i}\omega t}v (t)dt
\] 
denotes the Fourier transform of $v$.
\end{lemma}

\begin{proof}
The result is obvious for the case $\alpha=1$ if $v\in H^1 (\mathbb R)$. Next we show the assertion for the case $\alpha\in (0,1)$. Define a causal function $k_{+}^\alpha (t)$ with $\alpha\in (0,1)$ by
\begin{align*}
k_{+}^\alpha (t):=\left\{
\begin{array}{ll}
 \frac1{\Gamma (1-\alpha)}t^{-\alpha},\quad & t> 0,\\
0, &t\le0,
\end{array}\right.
\end{align*}
which satisfies 
\[
\partial^\alpha_t v (t)=\frac{1}{\Gamma (1-\alpha)}\int_0^t\partial_s v (s) (t-s)^{-\alpha}ds= (\partial_t v* k_+^\alpha) (t).
\] 
Based on the fact 
\[
\mathcal {F}[\partial^\alpha_t v] (\omega)=\mathcal{F}[\partial_t v] (\omega)\mathcal{F}[k_+^\alpha] (\omega)= ({\rm i}\omega)\hat v (\omega)\mathcal{F}[k_+^\alpha] (\omega),
\]
it suffices to show that the Fourier transform of $k_+^\alpha$ admits the following form (cf. \cite[Sec. 2.9.2]{Pod1998Fractional}):
\begin{align}\label{fourier_k}
\mathcal{F}[k_+^\alpha] (\omega) = \int_0^\infty {\frac{e^{-{\rm i}\omega t}t^{-\alpha}}{\Gamma (1-\alpha)}} dt=  ({\rm i}\omega)^{\alpha-1}.
\end{align}

In fact, let $U_R\subset\mathbb C$ be a simply connected open set with $\gamma_R$ being a closed curve shown in Figure \ref{fourier_figure}. It is clear to note that the mapping
\[
z\mapsto e^{-{\rm i}\omega z}z^{-\alpha},\quad z\in\mathbb C
\]
defines a holomorphic function in $U_R$. It follows from the Cauchy integral theorem that 
\begin{align}\label{fourier_proof1}
\int_{\gamma_R} e^{-{\rm i}\omega z}z^{-\alpha}dz = \left[\int_{1/R}^R+\int_{I_R}+\int_{{\rm i}R}^{{\rm i}/R}+\int_{I_{1/R}}\right]e^{-{\rm i}\omega z}z^{-\alpha}dz = 0,
\end{align}
where $I_R:=\{z\in\mathbb C:|z|=R\}$ and $I_{1/R}:=\{z\in\mathbb C:|z|=1/R\}$ denote the positively oriented a quarter section of the circle with radius $R$ (the dashed curve in Figure \ref{fourier_figure}) and the negatively oriented a quarter section of the circle with radius $1/R$ (the solid curve in Figure \ref{fourier_figure}), respectively.

\begin{figure}[h!]
  \centering
  \includegraphics[width=0.4\textwidth]{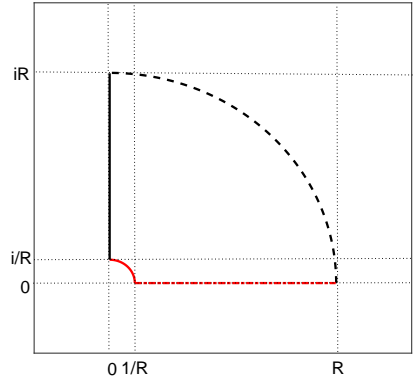}
  \caption{The integral contour $\gamma_R$}\label{fourier_figure}
\end{figure}

Clearly, we have  
\[
\lim_{R\to\infty}\left[\int_{I_R}+\int_{I_{1/R}}\right]e^{-{\rm i}\omega z}z^{-\alpha}dz = 0.
\] 
Taking the limit of \eqref{fourier_proof1} as $R\rightarrow +\infty$ leads to
\begin{align*}
\int_0^{+\infty}e^{-{\rm i}\omega z}z^{-\alpha}dz+\int_{+{\rm i}\infty}^0e^{-{\rm i}\omega z}z^{-\alpha}dz=0.
\end{align*}
Let ${\rm i}\omega z=s$. A simple calculation yields 
\begin{align*}
\int_0^{+\infty}e^{-{\rm i}\omega z}z^{-\alpha}dz&=\int_{0}^{+{\rm i}\infty}e^{-{\rm i}\omega z}z^{-\alpha}dz\\
&= ({\rm i}\omega)^{\alpha-1}\int_0^\infty e^{-s}s^{-\alpha}ds= ({\rm i}\omega)^{\alpha-1}\Gamma (1-\alpha),
\end{align*}
which completes the proof by noting \eqref{fourier_k}.
\end{proof}

\begin{remark}\label{rk}
Note that  $ ({\rm i}\omega)^\alpha$ is a multi-valued function when $\alpha$ is a fractional number. Throughout the paper, we define
\begin{align*}
 ({\rm i}\omega)^\alpha:=\left\{
\begin{array}{ll}
|\omega|^{\alpha}\exp\left (\frac{{\rm i}\pi\alpha}{2}\text{sgn} (\omega)\right),\quad & \omega\neq 0,\\
0,& \omega=0,
\end{array}\right.
\end{align*}
where $\text{sgn} (\cdot)$ denotes the sign function.
\end{remark}

\section{The direct problem}\label{Direct_Problem}

In this section, we discuss the direct problem. The well-posedness of the problem \eqref{question} and the regularity of its mild solution are investigated by studying an equivalent problem in the frequency domain. 

\subsection{The direct problem in the frequency domain}

Since the function $F$ satisfies $F (0)=0$, we denote by $\tilde F$ the zero extension of $F$ in $ (-\infty,0)$ and by $\hat F$ the Fourier transform of $\tilde F$. Consider the two-point boundary value problem of the stochastic differential equation
\begin{equation}\label{fourier}
\begin{cases}
  \partial_{xx}U (x,\omega)- ({\rm i}\omega)^\alpha U (x,\omega)=-\hat{F} (\omega)\dot{W}_x,\quad &x\in D,~\omega\in\mathbb R,\\
  \partial_xU (0,\omega)=0,\quad U (1,\omega)=0,\quad &\omega\in\mathbb R.\\
\end{cases}
\end{equation}

In the following, we deduce the Green function and present the well-posedness of \eqref{fourier}. 

\subsubsection{Green's function}

Let $s:= ({\rm i}\omega)^\alpha$ and denote by $g_\omega (x,y)$ the Green function of \eqref{fourier}. We consider two cases where $g_\omega$ takes two different expressions.   

If $\omega\neq0$, then $g_\omega$ satisfies
\begin{equation*}
\begin{cases}
\partial_{xx}g_\omega (x,y)-s g_\omega (x,y) =\delta (x-y),\quad& x,y\in D,\\
\partial_x g_\omega (0,y)=0, \quad g_\omega (1,y)=0,\quad& y\in D,\\
\end{cases}
\end{equation*}
where $\delta$ is the Dirac delta function. It is known from solving the second order ordinary differential equation with constant coefficients that $g_\omega (x,y)$ has the general form 
\[
g_\omega (x,y)=
\begin{cases}
A_1 (y)e^{- {\sqrt{s}} x}+B_1 (y)e^{ {\sqrt{s}} x},&\quad x<y,\\
A_2 (y)e^{- {\sqrt{s}} x}+B_2 (y)e^{ {\sqrt{s}} x},&\quad x>y,
\end{cases}
\]
where $A_i$ and $B_i$, $i=1,2,$ are to be determined. Using the boundary conditions
\[
\partial_x g_\omega (0,y)=0, \quad g_\omega (1,y)=0,
\]
and the continuity and jump conditions
\[
\lim_{x\to y^+}g_\omega (x,y)-\lim_{x\to y^-}g_\omega (x,y)=0,\quad\lim_{x\to y^+}\partial_xg_\omega (x,y)-\lim_{x\to y^-}\partial_xg_\omega (x,y)=1.
\]
we may easily obtain from Remark \ref{rk} that 
\begin{align}\label{greennotzero}
g_\omega (x,y)=\frac{  e^{ {\sqrt{s}} (x+y)}+e^{ {\sqrt{s}}|x-y|}-e^{ {\sqrt{s}} (2-x-y)}-e^{ {\sqrt{s}} (2-|x-y|)}}
{ {2\sqrt{s} (1+e^{2\sqrt{s}})}},\quad x,y\in D,
\end{align}
where 
\[
\sqrt{s}= ({\rm i}\omega)^{\frac{\alpha}{2}}=|\omega|^{\frac\alpha2}\exp\left (\frac{{\rm i}\pi\alpha}{4}\text{sgn} (\omega)\right),\quad \Re[\sqrt{s}]=|\omega|^{\frac\alpha2}\cos\left (\frac{\pi\alpha}4\right)>0. 
\]

If $\omega=0$, then $g_0 (x,y)$ solves
\begin{equation*}
\begin{cases}
\partial_{xx}g_0 (x,y)=\delta (x-y),\quad& x,y\in D,\\
\partial_x g_0 (0,y)=0, \quad g_0 (1,y)=0,\quad &y\in D.\\
\end{cases}
\end{equation*}
Similarly, we may solve the above two-point boundary value problem and obtain   
\begin{align}\label{greeniszero}
g_0 (x,y)=\max\{x,y\}-1,\quad x,y\in D.
\end{align}

\begin{lemma}\label{lm:g}
The Green function $g_\omega$ given in \eqref{greennotzero} and \eqref{greeniszero} satisfies the estimate
\[
\begin{cases}
\|g_0\|_{L^2 (D\times D)}=\frac16\quad &\text{if} ~ \omega=0, \\
\|g_{\omega}\|_{L^2 (D\times D)}\leq C |\omega|^{-\alpha}\quad &\text{if} ~ \omega\neq0,
\end{cases}
\]
where $C>0$ is a constant independent of $\omega$. 
\end{lemma}

\begin{proof}
If $\omega=0$, then it follows from a simple calculation that 
\begin{align*}
\|g_0\|_{L^2 (D\times D)}^2&=\int_0^1\left[ \int_0^x (x-1)^2dy + \int_x^1 (y-1)^2dy\right]dx\\
&=\int_0^1\left[x (1-x)^2+\frac{ (1-x)^3}{3}\right]dx=\frac16.
\end{align*}
If $\omega\neq0$, we get
\begin{align}\label{g_omega}
&\|g_{\omega}\|_{L^2 (D\times D)}^2=\int_D\int_D|g_\omega (x,y)|^2dxdy\notag\\
&=\frac{1}{|2\sqrt{s} (1+e^{2\sqrt{s}})|^2}\int_0^1\int_0^1\left|e^{\sqrt{s} (x+y)} + e^{\sqrt{s}|x-y|} - e^{\sqrt{s} (2-x-y)} - e^{\sqrt{s} (2-|x-y|)}\right|^2dxdy\notag\\
&\le  \frac{1}{|\sqrt{s} (1+e^{2 \sqrt{s}})|^2} \int_0^1\int_0^1 e^{2\Re[\sqrt{s}] (x+y)} + e^{2\Re[\sqrt{s}]|x-y|} + e^{2\Re[\sqrt{s}] (2-x-y)} + e^{2\Re[\sqrt{s}] (2-|x-y|)}dx dy\notag\\
&=\frac{1}{|\sqrt{s} (1+e^{2 \sqrt{s}})|^2}\frac{e^{4\Re[\sqrt{s}]}-1}{\Re[\sqrt{s}]}\notag\\
&=|\omega|^{-\alpha}\frac1{|\omega|^{\frac\alpha2}\cos (\frac{\pi\alpha}4)}\frac{e^{4|\omega|^{\frac\alpha2}\cos (\frac{\pi\alpha}4)}-1}{\left[e^{4|\omega|^{\frac\alpha2}\cos (\frac{\pi\alpha}4)}+1+2e^{2|\omega|^{\frac\alpha2}\cos (\frac{\pi\alpha}4)}\cos (2|\omega|^{\frac\alpha2}\sin (\frac{\pi\alpha}4))\right]}\notag\\
&=:|\omega|^{-\alpha}h\left (|\omega|^{\frac\alpha2}\cos\left (\frac{\pi\alpha}4\right)\right),
\end{align}
where the function $h$ is defined by
\[
h (k)=\frac{e^{4k}-1}{k\left[e^{4k}+1+2e^{2k}\cos\left (2k\tan (\frac{\pi\alpha}4)\right)\right]}.
\]
Clearly, $h$ is a nonnegative function for all $k>0$.

We claim that $h (k)$ is uniformly bounded for all $k>0$. On the one hand, due to the fact
\[
\lim_{k\to\infty}h (k)=\lim_{k\to\infty}\frac{1-e^{-4k}}{k\left[1+e^{-4k}+2e^{-2k}\cos\left (2k\tan (\frac{\pi\alpha}4)\right)\right]}=0,
\]
there exists a constant $C_0$ such that $h (k)<1$ for all $k>C_0$, which shows that $h$ is uniformly bounded on $ (C_0,\infty)$. On the other hand, by noting 
\[
\lim_{k\to0^+}h (k)=1,
\] 
we get that $h$ is also uniformly bounded on $(0,C_0]$ due to its smoothness, which completes the proof of the claim.

Using the estimates for $h$, we have from \eqref{g_omega} that there exists a constant $C>0$ independent of $\omega$ such that
\[
\|g_{\omega}\|_{L^2 (D\times D)}^2\le h\left (|\omega|^{\frac\alpha2}\cos\left (\frac{\pi\alpha}4\right)\right)|\omega|^{-\alpha}\leq C|\omega|^{-\alpha},
\]
which completes the proof.
\end{proof}

\subsubsection{Well-posedness}

Based on the Green function $g_\omega$,  we are able to show that the two-point boundary value problem of the stochastic differential equation \eqref{fourier} admits a unique mild solution. The following result gives the estimate of the mild solution, which plays an important role in the analysis of the time domain problem.

\begin{theorem}\label{tm:frequency}
Assume that $F\in H^1 (\mathbb R_+)$. Then the the stochastic differential equation \eqref{fourier} admits a unique mild solution given by 
\begin{equation}\label{fou_solution}
U (x,\omega)=-\hat{F} (\omega)\int_Dg_\omega (x,y)dW_y.
\end{equation}
Moreover, the solution $U$ satisfies the estimate
\begin{equation}\label{U_estimate}
\mathbb E\|{\rm i}\omega U\|_{L^2 (\mathbb R; L^2 (D))}^2\le C\|F\|_{H^1 (\mathbb R_+)}^2,
\end{equation}
where $C>0$ is a constant. 
\end{theorem}

\begin{proof}
The existence and uniqueness of the mild solution can be proved similarly as those in \cite{Bao2016Inverse}. We only show the proof for the estimate \eqref{U_estimate}. 

By It\^o's isometry, Lemma \ref{lm:g} and Parseval's identity, we get
\begin{align*}
\mathbb E\|{\rm i}\omega U\|_{L^2 (\mathbb R; L^2 (D))}^2
&=\int_{\mathbb R}\int_D|{\rm i}\omega\hat F (\omega)|^2\mathbb E\left|\int_Dg_\omega (x,y)dW_y\right|^2dxd\omega\\
&=\int_{\mathbb R}|{\rm i}\omega\hat F (\omega)|^2\|g_\omega\|_{L^2 (D\times D)}^2d\omega\\
&\lesssim\int_{\mathbb R}|\omega|^{-\alpha}|{\rm i}\omega\hat F (\omega)|^2d\omega\\
&\leq\int_{\{\omega:|\omega|\le1\}}|\hat F (\omega)|^2d\omega+\int_{\{\omega:|\omega|>1\}}|{\rm i}\omega\hat F (\omega)|^2d
\omega\\
&\leq\|\hat F\|_{L^2 (\mathbb R)}^2+\|{\rm i}\omega\hat F\|_{L^2 (\mathbb R)}^2=\|F\|_{H^1 (\mathbb R_+)},
\end{align*}
which completes the proof.
\end{proof}

Hereafter, the notation $a\lesssim b$ stands for $a\leq C b$, where $C$ is a positive constant whose value is not required but should be clear from the context.

\subsection{The direct problem in the time domain}

Using the result obtained for the equivalent problem in the frequency domain \eqref{fourier}, we are now at the position to show the well-posedness of the time domain problem \eqref{question}. 

\begin{theorem}
Assume that $F\in H^1 (\mathbb R_+)$. Then the initial boundary value problem of the stochastic differential equation
\eqref{question} admits a unique solution $u$ satisfying
\[
\mathbb E\|\partial_tu\|_{L^2 (D)}^2\le C\|F\|_{H^1 (\mathbb R_+)}^2\quad\forall\,t\in\mathbb R_+,
\]
where $C>0$ is a constant. 
\end{theorem}

\begin{proof}
Let
\[
\tilde u (x,t)=\mathcal{F}^{-1}[U (x,\cdot)] (t),\quad x\in D,~t\in\mathbb R,
\]
where $\mathcal{F}^{-1}$ denotes the inverse Fourier transform and $U$ is the mild solution of \eqref{fourier}. Define
\begin{equation}\label{eq:u}
u (x,t)=\tilde u (x,t)|_{t\in\mathbb R_+}.
\end{equation}
To show the existence of the solution to \eqref{question}, we prove that the function $u$ defined above is a solution of \eqref{question}.

Note that
\[
\tilde u (x,t)=\mathcal{F}^{-1}[U (x,\cdot)] (t)=-\int_{-\infty}^t\tilde F (s)\mathcal{F}^{-1}\left[\int_Dg_\omega (x,y)dW_y\right] (t-s)ds,
\]
where $\tilde F$ is the zero extension of $F$ on $ (-\infty,0)$ which is defined at the beginning of this section. Hence, $\tilde u$ is a causal function with $\tilde u (x,t)=0$ if $t\le0$, which implies that 
\begin{equation}\label{eq:initial}
u (x,0)=\tilde u (x,0)=0.
\end{equation}
Moreover, it follows from Parseval's identity and Theorem \ref{tm:frequency} that $\partial_t\tilde u\in L^2 (\Omega; L^2 (\mathbb R; L^2 (D))$ satisfies
\begin{equation}\label{eq:esti}
\mathbb E\|\partial_t\tilde u\|_{L^2 (\mathbb R; L^2 (D))}^2=\mathbb E\|{\rm i}\omega U\|_{L^2 (\mathbb R; L^2 (D))}^2\le C\|F\|_{H^1 (\mathbb R_+)}^2, 
\end{equation}
which indicates that the Caputo fractional derivative of $\tilde u$ with respect to $t$ is well-defined. 

Taking the inverse Fourier transform with respect to $\omega$ on both sides of \eqref{fourier}, we have
\begin{equation}\label{eq:model}
\begin{cases}
\partial_{xx}u (x,t)-\partial_t^\alpha u (x,t)=-F (t)\dot{W}_x,\quad &x\in D,~t\in\mathbb R_+,\\
\partial_xu (0,t)=0,\quad u (1,t)=0,\quad&t\in\mathbb R_+,
\end{cases}
\end{equation}
where we have used Lemma \ref{lm:frac} for the causal function $\tilde u$ and the fact $u=\tilde u|_{\mathbb R_+}$. The equation \eqref{eq:model}, together with the initial condition \eqref{eq:initial}, indicates that $u$ defined in \eqref{eq:u} is a solution of \eqref{question} satisfying the estimate \eqref{eq:esti}.

The uniqueness of the solution of \eqref{question} is obtained directly by the equivalence between the time domain problem \eqref{question} and the frequency domain problem \eqref{fourier}, as well as the well-posedness of \eqref{fourier} given in Theorem \ref{tm:frequency}.
\end{proof}

\section{The inverse problem}\label{Inverse_Porblem}

In this section, we address the inverse problem which is to reconstruct the diffusion coefficient $F$ of the source from the measured wave field $u(0,t)$ at the point $x=0$ for $t>0$.

\subsection{The reconstruction of $|\hat{F}|$.} 

First we consider the reconstruction of the modulus of the Fourier coefficients of $F$, and investigate the uniqueness and the issue of instability of the inverse problem.

\subsubsection{Uniqueness} 

It follows from \eqref{fou_solution} that the mean and variance of the solution $U (x,\omega)$ at $x=0$ satisfies 
\[
\mathbb E[U (0,\omega)]=0
\]
and 
\begin{eqnarray}\label{variance}
\mathbb{V}[U (0,\omega)]=\mathbb{E}[|U (0,\omega)|^2]=|\hat{F} (\omega)|^2\int_D|g_\omega (0,y)|^2dy,
\end{eqnarray}
where we have from \eqref{greennotzero} and \eqref{greeniszero} that 
\begin{align}\label{eq:g}
g_\omega (0,y)=\frac{e^{ {\sqrt{s}y}} - e^{ {\sqrt{s} (2-y)}}}{{\sqrt{s} (1 + e^{ {2\sqrt{s}}})}},\quad y\in D
\end{align}
if $\omega\neq0$ and 
\[
g_0 (0,y)=y-1,\quad y\in D
\]
if $\omega=0$. 

\begin{lemma}\label{inverse_uniqueness_lemma}
For any $\omega\in\mathbb R$, it holds
\begin{align*}
\int_D|g_\omega (0,y)|^2dy>0.
\end{align*}
\end{lemma}

\begin{proof}
For $\omega=0$, a simple calculation yields 
\[
\int_D|g (0,y)|^2dy=\int_0^1 (y-1)^2dy=\frac13.
\]
For $\omega\neq0$, it holds 
\begin{align*}
\int_D|g_\omega (0,y)|^2dy&=\frac{1}{| {\sqrt{s} (1+e^{ {2\sqrt{s}}})}|^2} \int_0^1\left|e^{ {\sqrt{s}y}}-e^{ {\sqrt{s} (2-y)}}\right|^2dy\\
&=\frac{1}{| {\sqrt{s} (1+e^{ {2\sqrt{s}}})}|^2} \int_0^1\left[e^{2\Re[\sqrt{s}]y}+e^{2\Re[\sqrt{s}] (2-y)} - 2e^{2\Re[\sqrt{s}]}\cos\left (2\Im[\sqrt{s}] (1-y)\right)\right]dy\\
&=\frac{1}{| {\sqrt{s} (1+e^{ {2\sqrt{s}}})}|^2}\left (\frac{e^{4\Re[\sqrt{s}]}-1}{2\Re[\sqrt{s}]} - \frac{e^{2\Re[\sqrt{s}]}\sin\left (2\Im[\sqrt{s}]\right)}{\Im[\sqrt{s}]}\right)\\
&\geq\frac{1}{| {\sqrt{s} (1+e^{ {2\sqrt{s}}})}|^2}\left (\frac{e^{4\Re[\sqrt{s}]}-1}{2\Re[\sqrt{s}]} - 2e^{2\Re[\sqrt{s}]}\right)\\
&=:\frac{1}{| {\sqrt{s} (1+e^{ {2\sqrt{s}}})}|^2}l_1\left (2\Re[\sqrt{s}]\right), 
\end{align*}
where we have used the fact ${\frac{\sin (x)}{x}}\leq1$ for any $x\in\mathbb{R}$, and the function $l_1$ is defined by
\[
l_1 (k)=\frac{e^{2k}-1}{k}-2e^{k},\quad k>0.
\]
It then suffices to show that $l_1 (k)>0$ for all $k>0$. Note that
\begin{align*}
l_1' (k)=\frac{ (2k-1)e^{2k}-2k^2e^k+1}{k^2}=:\frac{l_2 (k)}{k^2},
\end{align*}
where $l_2 (k)= (2k-1)e^{2k}-2k^2e^k+1$.
It is easy to check that $l_2 (k)>0$ and hence $l_1' (k)>0$ for all $k>0$. In fact, by noting that
\[
l_2' (k)= (4ke^k-2k^2-4k)e^k> (4k (1+k)-2k^2-4k)e^k>0\quad\forall\,k>0,
\]
we get
\[
l_2 (k)>\lim_{k\to0^+}l_2 (k)=0.
\]
Hence the function $l_1$ is increasing for all $k>0$ and satisfies
\[
l_1 (k)>\lim_{k\to0^+}l_1 (k)=\lim_{k\to0^+}\left[\frac{e^{2k}-1}{k}-2e^k\right]=0,
\]
which completes the proof.
\end{proof}

\begin{theorem}\label{inverse_uniqueness_thm}
Assume that $F\in H^1 (\mathbb R_+)$. Then the modulus $|\hat{F} (\omega)|$ can be uniquely determined by the data $\mathbb{V}[U (0,\omega)]$.
\end{theorem}

\begin{proof}
It follows from \eqref{variance} and Lemma \ref{inverse_uniqueness_lemma} that 
\[
|\hat F (\omega)|=\left (\frac{\mathbb V[U (0,\omega)]}{\int_D|g_\omega (0,y)|^2dy}\right)^{\frac12}\quad\forall\,\omega\in\mathbb R,
\]
which implies the uniqueness of the reconstruction.
\end{proof}

\subsubsection{Instability}

By Theorem \ref{inverse_uniqueness_thm}, the inverse problem admits a unique solution but it lacks stability due to the fact that the denominator $\int_D|g_\omega (0,y)|^2dy$ goes to zero as $\omega\to\infty$, which is stated in the following theorem.

\begin{theorem}\label{inverse_stability_thm}
For any fixed $\omega\neq0$, it holds
\begin{align*}
\int_D|g_\omega (0,y)|^2dy\lesssim|\omega|^{-\alpha}.
\end{align*}
\end{theorem}

\begin{proof}
By \eqref{eq:g}, we have 
\begin{align*}
\int_D|g_\omega (0,y)|^2dy&=\frac{1}{| {\sqrt{s} (1+e^{ {2\sqrt{s}}})}|^2} \int_0^1\left|e^{ {\sqrt{s}y}} - e^{ {\sqrt{s} (2-y)}}\right|^2dy\notag\\
&\leq\frac{2}{| {\sqrt{s} (1+e^{ {2\sqrt{s}}})}|^2} \int_0^1\left (\left|e^{ {\sqrt{s}y}}\right|^2 + \left|e^{ {\sqrt{s} (2-y)}}\right|^2\right)dy\notag\\
&=\frac{2}{|{\sqrt{s} (1+e^{ {2\sqrt{s}}})}|^2} \int_0^1\left (e^{2\Re[\sqrt{s}]y}+e^{2\Re[\sqrt{s}] (2-y)}\right)dy\notag\\
&=\frac{1}{|{\sqrt{s} (1+e^{ {2\sqrt{s}}})}|^2}\frac{e^{4\Re[\sqrt{s}]}-1}{\Re[\sqrt{s}]}.
\end{align*}
Following from the same estimate as that of \eqref{g_omega}, we get
\[
\int_D|g_\omega (0,y)|^2dy\le|\omega|^{-\alpha}h\left (|\omega|^{\frac\alpha2}\cos\left (\frac{\pi\alpha}4\right)\right)\lesssim|\omega|^{-\alpha}, 
\]
which completes the proof. 
\end{proof}

\subsection{Phase retrieval}\label{sec:phaseretrieval}

In this section, we discuss the phase retrieval problem. More precisely, we aim to numerically reconstruct $|F|$ from $|\hat F|$, where the former is the modulus of the diffusion coefficient $F$ of the source in the time domain and the latter is the modulus of the diffusion coefficient $\hat F$ in the frequency domain. 

To introduce the numerical methods for the phase retrieval, we begin with presenting a discrete version of the phase retrieval problem, and then introduce an additional measurement based framework named PhaseLift \cite{Candes2015Phase}, which is adopted to solve the phase retrieval problem.

\subsubsection{Discrete phase retrieval problem}

Let $\bm{x}= (x_1,...,x_N)^\top\in\mathbb C^N$ be a signal of length $N$, and $\bm{y}= (y_1,...,y_N)^\top\in\mathbb C^N$ be its $N$-point discrete Fourier transform (DFT). 
Denote by $\bm{f}^{(m)}$ the conjugate of the $m$th column of the $N$-point DFT matrix, i.e.,
\begin{equation}\label{eq:fm}
\bm f^{(m)}=\left(f^{(m)}_1,\cdots,f^{(m)}_N\right)^\top:=\left(1,e^{{\rm i}\frac{ 2\pi(m-1) }{N}},\cdots,e^{{\rm i}\frac{2\pi (m-1) (N-1)}{N}}\right)^\top.
\end{equation}
Then it is easy to check that $y_m=\langle\bm f^{(m)},\bm x\rangle$, where $\langle\cdot,\cdot\rangle$ is the complex inner product defined by
\[
\langle\bm x,\bm y\rangle:=\sum_{n=1}^N\overline{x_n}y_n.
\]
The discrete phase retrieval problem is formulated as follows:
\begin{align}\label{prob1}
\begin{split}
\text{find} \qquad &\bm{x}\\
\text{subject to} \qquad &z_m:=|\langle\bm{f}^{(m)},\bm{x}\rangle|^2, \quad m=1,\cdots,N.
\end{split}
\end{align}

Pioneered by Gerchberg in \cite{Gerchberg1972Practical}, earlier approaches for the phase retrieval problem are based on alternating projections and can be reformulated as the following least-squares problem:
\begin{align*}
\min_{\bm{x}}\sum^{M}_{m=1}\left (z_m-|\langle\bm{f}^{(m)},\bm{x}\rangle|^2\right)^2.
\end{align*}
The algorithm requires oversampling by using an $M$-point DFT with $M>N$. It attempts to minimize the above non-convex objective by starting with a random initialization and iteratively imposing the time domain and Fourier magnitude constrains using projections. However, since the projections are taken between a convex set and  a non-convex set, the solution usually converges to a local minimum, which leads to the limited recovery ability of the algorithm even in the deterministic setting.

Recent frameworks to attack the ill-posedness of the phase retrieval problem can be broadly classified into two categories: (1) developing a modified model with a prior information; (2) taking more magnitude measurements. The former aims to reduce the number of unknowns by assuming some a prior information of the signal, such as the support constraints \cite{Marshesini2007Unified,Fienup1982Phase,Chen2007Application}, positivity and real-valuedness \cite{Fienup1982Phase,Fienup1978Reconstruction}, or sparsity \cite{Jaganathan2017Sparse,Shechtman2014GESPAR,Mukherjee2012iterative}. Depending on the applications, the latter can be done in various ways, which include the use of masks \cite{Johnson2008Coherent}, optical gratings \cite{Loewen2018Diffraction}, oblique illuminations \cite{Faridian2010Nanoscale}, or short-time Fourier transform magnitude measurements which utilize overlap between adjacent short-time sections\cite{Trebino2012Frequency,Rodenburg2008Ptychography}.

\subsubsection{PhaseLift}

Based on the semi-definite programming method, the PhaseLift is an effective approach to solve the phase retrieval problem. It has been shown that the PhaseLift may yield robust solutions to various quadratic-constraints. Since phase retrieval results are in quadratic constraints, the PhaseLift is adopted to handle our inverse problem. 

There are two main ingredients in this method: multiple structured illuminations and lifting. For the self-contained purpose, we briefly introduce the two components of this method. The details may be found in \cite{Candes2015Phase}. 

First we comment on the multiple structured illuminations. Let $\bm{x}= (x_1,...,x_N)^\top\in\mathbb C^N$ be the object of interest, and assume that the illumination schemes, which collect the diffraction pattern of the modulated object $\{w_n x_n\}_{n=1,\cdots,N}$, are available, where the pattern $\bm{w}= (w_1,...,w_N)^\top\in\mathbb C^N$ may be selected by the user. 
There are several ways to implement the recovery in practice, such as masking, optical grating, and oblique illuminations. It is usually preferred to adopt the approach in which fewer diffraction patterns are required for a stable recovery.
 
Next we present the lifting. Based on the diffraction pattern $\bm w$ mentioned above, suppose that we have quadratic measurements of the form
\begin{align*}
\bm{b}=(b_1,\cdots,b_N)^\top:=(|\langle \bm{a}^{(1)},\bm{x}\rangle|^2,\cdots,|\langle \bm{a}^{(N)},\bm{x}\rangle|^2)^\top,
\end{align*}
where $\bm a^{(m)}$ can be chosen based on the diffraction pattern $\bm w$ in the following form:
\[
\bm{a}^{(m)}=\left(f^{(m)}_1w_1,\cdots,f^{(m)}_Nw_N\right)^\top,
\]
where $\bm f^{(m)}=\left(f^{(m)}_1,\cdots,f^{(m)}_N\right)^\top$ is defined in \eqref{eq:fm}. The phase retrieval problem is then transformed to the feasibility problem: 
\begin{align*}
\begin{split}
\text{find} \qquad &\bm{x}\\
\text{subject to} \qquad &b_m:=|\langle\bm{a}^{(m)},\bm{x}\rangle|^2, \quad m=1,\cdots,N.
\end{split}
\end{align*}

The core idea of the PhaseLift is to embed the signal $\bm{x}$ into a higher dimensional space by using the transformation $\bm{X}=\bm{xx}^*$. Next, we lift up and interpret the signal $\bm{x}$ to this rank-one matrix $\bm{X}$. Denoting $\bm A^{(m)}:=\bm{a}^{(m)}(\bm{a}^{(m)})^*$, we get
\begin{align}\label{eq:bm}
b_m=|\langle \bm{a}^{(m)},\bm{x}\rangle|^2 = \bm{x}^*\bm{a}^{(m)}(\bm{a}^{(m)})^*\bm{x} =\text{Tr}(\bm A^{(m)}\bm{xx}^*),
\end{align}
where Tr$(\cdot)$ denotes the trace of a matrix. Let $\mathcal {A}$ be the linear operator mapping any positive semi-definite matrices $\bm X$, denoted by $\bm X\succeq0$, into 
\begin{equation}\label{eq:A}
\mathcal{A} (\bm{X}):=(\text{Tr}(\bm{A}^{(1)}\bm{X}),\cdots,\text{Tr}(\bm{A}^{(N)}\bm{X}))^\top.
\end{equation} 

Based on the above notation, the phase retrieval problem is equivalent to
\begin{align}\label{lift1}
\begin{split}
\text{find} \qquad &\bm{X}\\
\text{subject to} \qquad &\mathcal{A} (\bm{X})=\bm{z},\\
&\bm{X}\succeq0,\\
&\text{rank}(\bm{X})=1,
\end{split}
\end{align}
where $\bm z=(z_1,\cdots,z_N)^\top$ is defined in \eqref{prob1}. Consequently, the phase retrieval reduces to finding a rank-one positive semi-definite matrix $\bm{X}$ which satisfies these affine measurement constraints. Equivalently, the phase retrieval problem can be formulated to 
\begin{align}\label{lift2}
\begin{split}
\text{minimize} \qquad &\text{rank}(\bm{X})\\
\text{subject to} \qquad &\mathcal{A} (\bm{X})=\bm{z},\\
&\bm{X}\succeq0.
\end{split}
\end{align}

The equivalence between \eqref{lift1} and \eqref{lift2} is obvious since  $\bm{b}=\mathcal{A} (\bm{xx}^*)$ according to \eqref{eq:bm} and it has a rank-one solution. After solving \eqref{lift2}, we factorize the rank-one solution $\bm{X}$ as $\bm{xx}^*$, which then leads to the solution of the phase retrieval problem. Therefore, our inverse problem is equivalent to a rank-minimization problem over an affine slice of the positive semi-definite cone. Furthermore, it is reduced to a problem of low-rank matrix completion or matrix recovery, which is a classical optimization problem that has gained tremendous attention in recent years \cite{Candes2009Exact,Candes2010Power}.

However, it is known that the rank-minimization problem \eqref{lift2} is NP-hard. The trace norm is thus explored as a convex surrogate\cite{Beck1998Computational,Mesbahi1997Rank} for the rank functional in \eqref{lift2}, which gives the following semi-definite programming problem:
\begin{align}\label{lift3}
\begin{split}
\text{minimize} \qquad &\text{Tr}(\bm{X})\\
\text{subject to} \qquad &\mathcal{A} (\bm{X})=\bm{z},\\
&\bm{X}\text{ is Hermitian positive semi-definite}.
\end{split}
\end{align}
The semi-definite programming problem \eqref{lift3} is convex, and there exists a wide choice of numerical solvers including the popular Nesterov accelerated first-order method \cite{Nesterov2013Intro}. As far as the relationship between \eqref{lift2} and \eqref{lift3} is concerned, it is beyond the scope of this work and we refer to \cite{Candes2009Exact,Candes2010Power} for the detailed discussion. 

\subsection{Numerical method}\label{num_phase}

In this work, all the numerical algorithms are implemented in MATLAB by modifying the templates for first-order conic solvers (TFOCS) \cite{Becker2011TFOCS}. The TFOCS is a library of MATLAB files which are designed to facilitate the construction of optimal first-order methods for a variety of convex optimization problems including the semi-definite programming problem \eqref{lift3} considered in this paper.

To illustrate how we actually handle \eqref{lift3}, we next briefly introduce the formulation and implementation of a class of optimal first-order methods to solve the following general convex optimization problem: 
\begin{align}\label{general_optimi_prob}
\text{minimize}\quad &\phi (\bm x):=g(\bm x)+h (\bm x),
\end{align}
where $g:\mathbb R^N\to\mathbb R$ is convex and smooth and $h:\mathbb R^N\to\mathbb R$ is convex. 

Consider a class of first-order methods to solve \eqref{general_optimi_prob} based on iterations with a generalized projection (cf. \cite{Becker2011TFOCS}):
\begin{align}\label{1st_order_approx}
\bm x_{k+1}=\mathop{\arg\min}_{\bm x}\left[g(\bm x_k)+\langle\nabla g(\bm x_k),\bm x-\bm x_k\rangle + h (\bm x) + \frac{1}{2t_k}\|\bm x-\bm x_k\|^2\right],
\end{align}
where $\|\cdot\|$ is a chosen norm and $t_k$ is the step size. The global convergence of \eqref{1st_order_approx} is guaranteed if $t_k$ is bounded away from zero and $g(\bm x_{k+1})$ has the following upper bound
\begin{align}\label{conver_bound}
g(\bm x_{k+1})\leq g(\bm x_k)+\langle\nabla g(\bm x_k),\bm x_{k+1}-\bm x_k\rangle + \frac{1}{2t_k}\|\bm x_{k+1}-\bm x_k\|^2,
\end{align}
which can be accomplished by assuming that $\nabla g$ satisfies a generalized Lipschitz criterion
\begin{align}\label{eq:nablag}
\|\nabla g(\bm x)-\nabla g(\bm y)\|_{*}\leq L\|\bm x-\bm y\|
\end{align}
for any $\bm x,\bm y$ belonging to the domain of $\phi$. Here, $\|\cdot\|_{*}$ denotes the dual norm of the norm $\|\cdot\|$ defined by
\[
\|g\|_*=\sup\{\langle h,g\rangle:\|h\|\leq 1\}.
\] 
Then the bound of \eqref{conver_bound} is assured for any $t_k\leq L^{-1}$ under the assumption \eqref{eq:nablag}, which leads to the convergence 
\[
\left|\phi(\bm x_{k})-\inf_{\bm x}\phi(\bm x)\right|\le\epsilon
\]
in $O(L/\epsilon)$ iterations for a simple algorithm, known as the forward-backward algorithm or proximal gradient descent \cite{FM81}, based on \eqref{1st_order_approx}. 

Optimal or accelerated first-order methods are able to improve the bound of number of iterations to $O(\sqrt{L/\epsilon})$, and have been studied by many researchers in the past decades (cf. \cite{N88,T08}). The TFOCS implements a variety of the optimal first-order variants based on a variation of the method described by Nesterov in \cite{Nesterov2013Intro}. One variant, the Auslender and Teboulle method\cite{AT2006Interior}, is described as follows:
\begin{align}\label{AT_method}
\begin{split}
&\bm y_k =  (1-\theta_k)\bm x_k + \theta_k\overline{\bm x}_k,\\
&\overline{\bm x}_{k+1} = \mathop{\arg\min}_{\bm x}\langle \nabla g(\bm y_k),\bm x\rangle + \frac{1}{2}\theta_kL\|\bm x-\overline{\bm x}_k\|^2 + h (\bm x),\\
&\bm x_{k+1} =  (1-\theta_k)\bm x_k + \theta_k\overline{\bm x}_k+1,\\
&\theta_{k+1} = \frac{2}{1+\sqrt{1+4/\theta_k^2}},
\end{split}
\end{align}
where $\bm x_0$ is chosen in the domain of $\phi$, $\overline{\bm x}_0=\bm x_0$, and $\theta_0=1$. 
Here, ${\theta_k}$ is usually referred to as the accelerated parameter.

\section{Numerical Experiments}\label{Numerical_Exps}

In this section, we discuss the implementation for solving the direct and inverse random source problems, and report two numerical examples to demonstrate the effectiveness of the proposed methods.

\subsection{Numerical discretization of the direct problem}

To avoid the inverse crime, we employ the finite difference method to discretize the initial boundary value problem of the stochastic differential equation
\begin{equation*}
\begin{cases}
  \partial^{\alpha}_t u (x,t)-\partial_{xx}u (x,t)=F (t)\dot{W_x},\\
  u (x,0)=0,\\
  \partial_xu (0,t)=0, \quad u (1,t)=0
\end{cases}
\end{equation*}
over the interval $(0,1)\times(0,T)$ for some $T>0$.

Define the partition of the time and spatial intervals with nodes
\begin{align*}
t_n=nh_t,~ n=0,1,...,N_t,\quad x_i=i h_x, ~ i=0,1,...,N_x,
\end{align*}
where $h_t=T/N_t$ and $h_x=1/N_x$, respectively. Let $u_i^n$ be the numerical approximation to $u (x_i,t_n)$. The fractional derivative $\partial_t^\alpha u $ at $(x_i,t_n)$ is approximated by
\begin{align*}
\partial_t^\alpha u (x_i,t_n)&=\frac{1}{\Gamma (1-\alpha)}\int_0^{t_n}\frac{\partial u (x_i,s)}{\partial s}\frac{1}{ (t_n-s)^\alpha}ds\\
&=\frac{1}{\Gamma (1-\alpha)}\sum_{j=1}^{n}\int_{t_{j-1}}^{t_j}\frac{\partial u (x_i,s)}{\partial s}\frac{1}{(t_n-s)^\alpha} ds\\
&\approx \frac{1}{\Gamma (1-\alpha)}\sum_{j=1}^{n}\frac{u_i^j-u_i^{j-1}}{h_t}\int_{t_{j-1}}^{t_j} (t_n-s)^{-\alpha}ds\\
&=\frac{1}{\Gamma (1-\alpha)}\frac{1}{h_t^\alpha}\frac{1}{1-\alpha}\sum_{j=1}^{n} (u_i^j-u_i^{j-1})[ (n-j+1)^{1-\alpha}- (n-j)^{1-\alpha}]\\
&=\frac{1}{\Gamma (2-\alpha)} \frac{1}{h_t^\alpha}\Big[u_i^n+\sum_{j=1}^{n-1}u_i^j\left ( (n-j+1)^{1-\alpha}-2 (n-j)^{1-\alpha}+ (n-j-1)^{1-\alpha}\right)\\ 
&\quad - u_i^0\left(n^{1-\alpha}- (n-1)^{1-\alpha}\right)\Big].
\end{align*}
The second order derivative $\partial_{xx}u$ at $(x_i,t_n)$ is approximated by using the central difference method
\begin{align*}
\partial_{xx}u (x_i,t_n)\approx\frac{u_{i+1}^n-2u_i^n+u_{i-1}^n}{h_x^2}.
\end{align*}
The white noise $\dot W_x$ at $x=x_i$ is approximated by the increment $[W(x_{i+1})-W(x_{i})]/h_x$ of the Brownian motion $W$ satisfying
\[
\frac{W(x_{i+1})-W(x_{i})}{h_x}\overset{d}{=}\frac{\xi_i}{\sqrt{h_x}},
\]
where $a\overset{d}{=}b$ means that the random variables  $a$ and $b$ have the same distribution and $\{\xi_i\}_{i=0}^{N_x-1}$ is a set of independent and identically distributed standard normal random variables, denoted by $\xi_i\sim N(0,1)$.

Using the above approximations, we obtain the following implicit scheme: 
\begin{eqnarray*}\label{discretization}
&&\frac{1}{\Gamma (2-\alpha)h_t^\alpha} \Big[u_i^n+\sum_{j=1}^{n-1}u_i^j\left ( (n-j+1)^{1-\alpha}-2 (n-j)^{1-\alpha}+ (n-j-1)^{1-\alpha}\right)\notag\\ 
&&- u_i^0 \left(n^{1-\alpha}- (n-1)^{1-\alpha}\right)\Big]-\frac{u_{i+1}^n-2u_i^n+u_{i-1}^n}{h_x^2}
= F (t_n)\frac{W (x_{i+1})-W (x_{i})}{h_x}.
\end{eqnarray*}
By denoting $\sigma= \frac{h_x^2}{\Gamma (2-\alpha)h_t^\alpha}$, the above scheme can be rewritten as
\begin{align}\label{full_discrete}
&\quad -u_{i+1}^n+ (\sigma+2)u_i^n-u_{i-1}^n\notag\\
&=h_xF (t_n)[W (x_{i+1})-W (x_{i})]-\sigma u_i^0 \left(n^{1-\alpha}- (n-1)^{1-\alpha}\right)\notag\\
&\quad -\sigma\sum_{j=1}^{n-1}u_i^j \left( (n-j+1)^{1-\alpha}-2 (n-j)^{1-\alpha}+ (n-j-1)^{1-\alpha}\right)\notag\\
&=h_xF (t_n)[W (x_{i+1})-W (x_{i})]\notag\\
&\quad -\sigma\sum_{j=1}^{n-1}u_i^j \left( (n-j+1)^{1-\alpha}-2 (n-j)^{1-\alpha}+ (n-j-1)^{1-\alpha}\right)\notag\\
&=: G_i^n
\end{align}
for $n=1,\cdots,N_t$ and $i=0,1,\cdots,N_x-1$, where we have used the initial condition 
\begin{equation*}
u_i^0=0\quad\forall\,i=0,\cdots,N_x.
\end{equation*}

Moreover, according to the boundary condition $\partial_xu (0,t)=0$ in \eqref{question} and noting that
\begin{align*}
\partial_xu (0,t)=\lim_{h_x\rightarrow0}\frac{u (h_x,t)-u (-h_x,t)}{2h_x}=0,
\end{align*}
we define
\begin{align*}
u_{-1}^n:=u_1^n\quad\forall\,n=0,\cdots,N_t.
\end{align*}
Hence, when $i=0$, \eqref{full_discrete} leads to
\begin{equation}\label{eq:boundary1}
-2u_1^n+ (\sigma+2)u_0^n=G_0^n\quad\forall\,n=1,\cdots,N_t.
\end{equation}
By the boundary condition $u(1,t)=0$ in \eqref{question}, the following boundary condition is taken into account to the numerical scheme:
\begin{equation*}
u_{N_x}^n=0\quad\forall~n=0,\cdots,N_t,
\end{equation*}
and thus \eqref{full_discrete}, when $i=N_x-1$, turns to be
\begin{equation}\label{eq:boundary2}
(\sigma+2)u_{N_x-1}^n-u_{N_x-2}^n=G_{N_x-1}^n\quad\forall~n=1,\cdots,N_t.
\end{equation}

Combining \eqref{full_discrete} with \eqref{eq:boundary1}--\eqref{eq:boundary2}, we conclude that the numerical scheme has the following compact matrix form: 
\begin{equation*}\label{eq:scheme}
\left[
\begin{matrix}
  \sigma+2 & -2       & 0        & \cdots   & 0 \\
    -1       & \sigma+2 & -1       & \cdots   & 0 \\
    0        & -1       & \sigma+2 & \cdots   & 0 \\
    \vdots   & \cdots   & \ddots   & \ddots   & \vdots \\
    0        & \cdots   & -1       & \sigma+2 & -1     \\
    0        & \cdots   & 0        & -1       & \sigma+2 \\
\end{matrix}
\right]
\left[
\begin{matrix}
u_0^n\\
u_1^n\\
u_2^n\\
\vdots\\
\vdots\\
u_{N_x-1}^n
\end{matrix}
\right]
=\left[
\begin{matrix}
G_0^n\\
G_1^n\\
G_2^n\\
\vdots\\
\vdots\\
G_{N_x-1}^n
\end{matrix}
\right]\quad\forall\, n=1,\cdots,N_t.
\end{equation*}

\subsection{Implementation of the PhaseLift}

For the problem \eqref{question}, the modulus of the Fourier coefficient $|\hat{F}|$ can be reconstructed from \eqref{variance} by utilizing the variance $\mathbb V[U(0,\omega)]$ of the Fourier transform of the field $u(0,t)$.  To solve the phase retrieval problem numerically, i.e., to uniquely recover the source $\bm f:=(F(t_1),\cdots,F(t_{N_t}))^\top$ at discrete temporal nodes, more measurements are usually required and the PhaseLift with masks introduced in Section \ref{sec:phaseretrieval} is used. 

Let $\bm M_1$ and $\bm M_2$ be two masks with $\bm M_1$ being the $N_t\times N_t$ identity matrix and $\bm M_2$ being a random binary mask, i.e., an $N_t\times N_t$ diagonal matrix whose entries are randomly chosen as $0$ or $1$.
Using these two masks to perform the structured illumination with sources $\bm M_i\bm f$, we can obtain in total $2N_t$ intensity measurements $\mathcal{A}(\bm f\bm f^*)$, where the operator $\mathcal{A}$ is defined in \eqref{eq:A} with $N=2N_t$, the vectors $\{\bm a^{(m)}\}_{m=1}^{2N_t}$ involved in the operator $\mathcal{A}$ are defined by
\[
\bm a^{(m)}=\begin{cases}
\bm M_1\bm f^{(m)},&\quad m=1,\cdots,N_t,\\
\bm M_2\bm f^{(m-N_t)},&\quad m=N_t+1,\cdots,2N_t,
\end{cases}
\]
and $\bm f^{(m)}$ is given in \eqref{eq:fm}.

To recover the source $\bm f$, it then suffices to solve the convex optimization problem introduced in Section \ref{num_phase} based on the objective functional 
\begin{align*}
\phi (\bm{f})=g(\bm f)+h(\bm f)=\frac{1}{2}\|\bm b-\mathcal{A} (\bm f\bm f^*)\|_2^2+\lambda\text{Tr}(\bm{f}\bm f^*)+h (\bm{f}),
\end{align*}
where $\|\cdot\|_2$ denotes the Euclidean norm and $h$ is the indicator function defined by
\begin{align*}
h (\bm{f})=
\begin{cases}
0\quad &\text{if }\bm{f}\bm f^*\text{ is Hermitian positive semi-definite,}\\
\infty\quad &\text{else}.
\end{cases}
\end{align*}
Here, $\bm b$ is the quadratic Fourier modulus obtained by the measurement $\{u_0^n\}_{n=1}^{2N_t}$ at $x=0$ perturbed by the sources $[(\bm M_1\bm f)^\top,(\bm M_2\bm f)^\top]^\top$ with masks. More precisely, $\bm b=(b_1,\cdots,b_{2N_t})^\top$ has entries
\[
b_n=\frac{\mathbb V[U_0^{n}]}{\int_0^1|g_{\omega_n}(0,y)|^2dy},\quad n=1,\cdots,2N_t,
\]
where $(U_0^1,\cdots,U_0^{2N_t})^\top$ is the $2N_t$-point DFT of the measurement $(u_0^1,\cdots,u_0^{2N_t})^\top$ and $\omega_n=2\pi nh_t/N_t$. 

The Auslender and Teboulle method \eqref{AT_method} is applied to solve the above problem, which is terminated when the relative residual error of our reconstructed result $\tilde{\bm f}$ is less than a fixed tolerance, i.e., $||\mathcal{A} (\tilde{\bm f}\tilde{\bm f}^*)-\bm{b}||_2\leq 10^{-6}||\bm{b}||_2$. Note that the solution $\tilde{\bm f}=(\tilde f_1,\cdots,\tilde f_{N_t})^\top$ is unique only up to the global phase, i.e., $(|c\tilde f_1|^2,\cdots,|c\tilde f_{N_t}|^2)^\top=(|F(t_1)|^2,\cdots,|F(t_{N_t})|^2)^\top$ is unique, where $c$ is a complex scalar satisfying $|c|=1$. In particular, if $F(t)$ is a real-valued nonnegative function, then the recovery is unique. As a result, the absolute value $|F|$ is recovered in the following two numerical examples.

We set $N_t=65$ and $N_x=100$, and use the total number of 1000 sample paths to approximate the variance of the solution. In addition, the data is assumed to be polluted by a uniformly distributed noise with the noise level $\sigma=0.05$. In the first example, we set the source function $F(t)=\sin (t)\exp (-t/6)$ and $T=4\pi$, and in the second example, the source function is taken to be $F(t)=\sin (2t)\cos (3t)$ with $T=\pi$. We present the results for two different cases: $\alpha=0.4$ and $\alpha=0.8$ in both of the examples, which are shown in Figures \ref{num_soluton_f1} and \ref{num_soluton_f2}, respectively. We comment that (1) the larger the $\alpha$ is, the more regular of the solution is for the direct source problem; (2) the smaller the $\alpha$ is, the more stable of the solution for the inverse source problem. It can be observed from the numerical experiments that the proposed method is effective and robust to reconstruct the modulus of the source functions. 

\begin{figure}[htbp]
  \centering
  \subfigure[]{\includegraphics[width=0.45\textwidth]{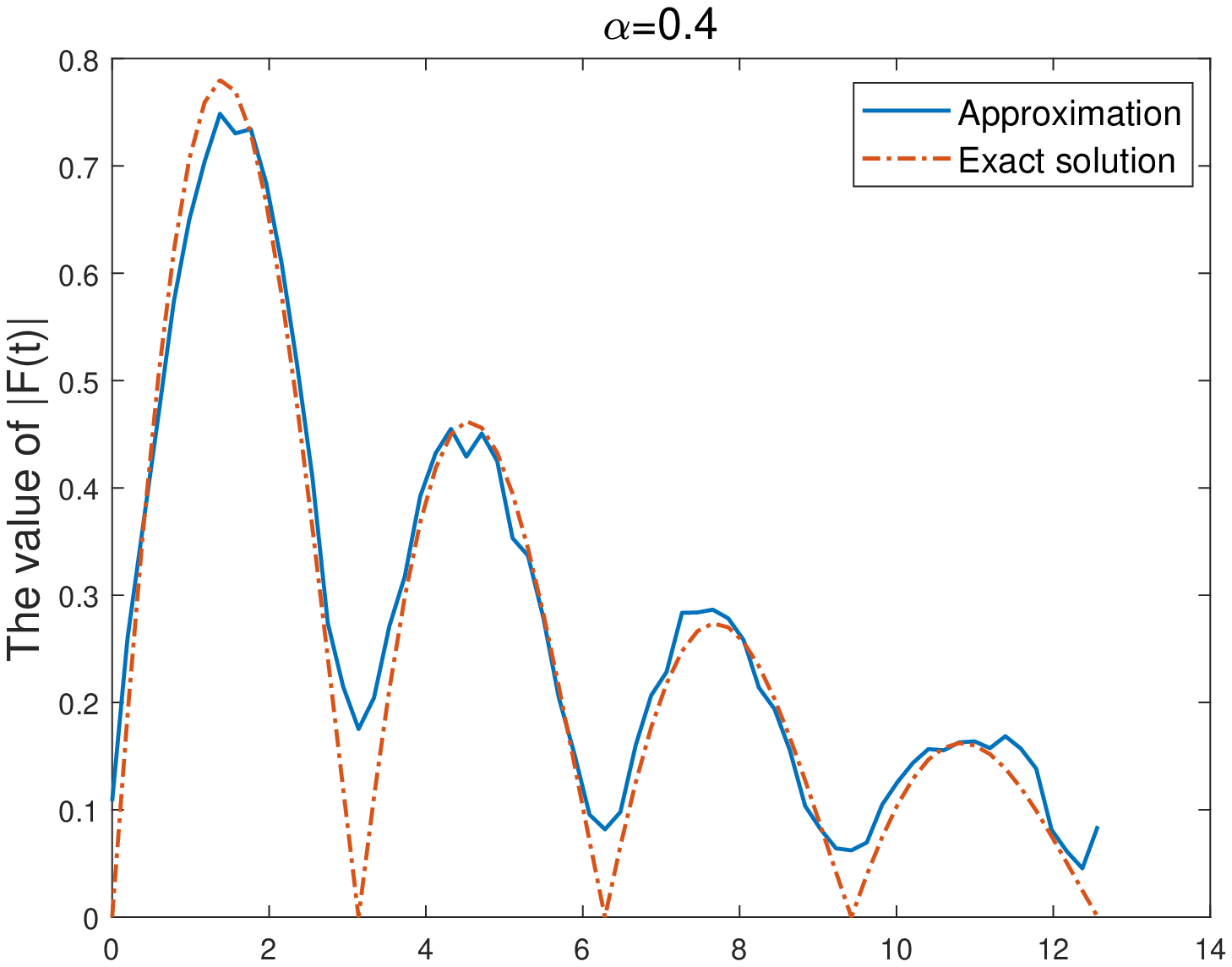}}
  \subfigure[]{\includegraphics[width=0.45\textwidth]{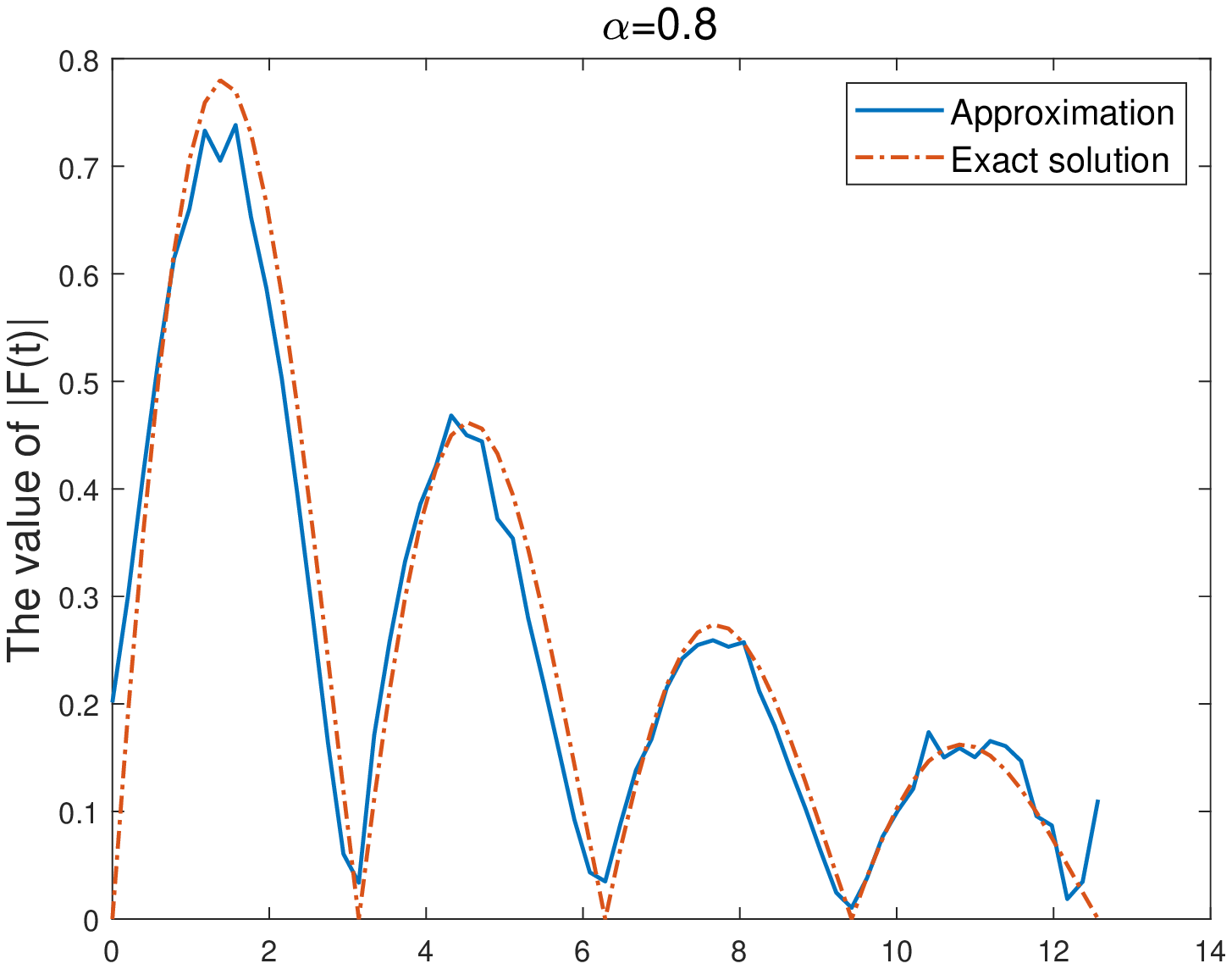}}\\
  \caption{Example 1: the exact solution is plotted against the reconstructed solution of $|F|$. (left) $\alpha =0.4$; (right) $\alpha=0.8$.}\label{num_soluton_f1}
\end{figure}

\begin{figure}[htbp]
  \centering
  \subfigure[]{\includegraphics[width=0.45\textwidth]{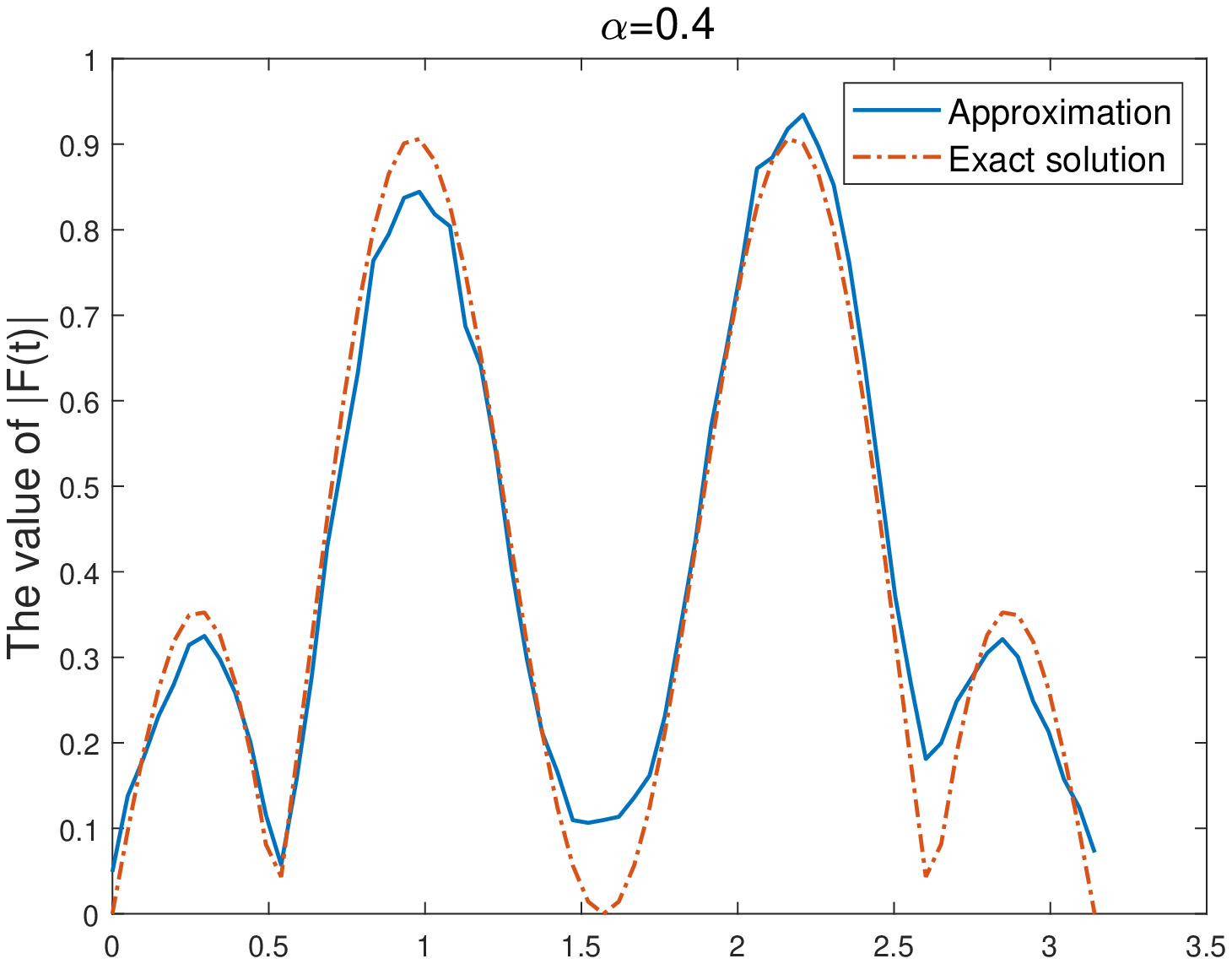}}
  \subfigure[]{\includegraphics[width=0.45\textwidth]{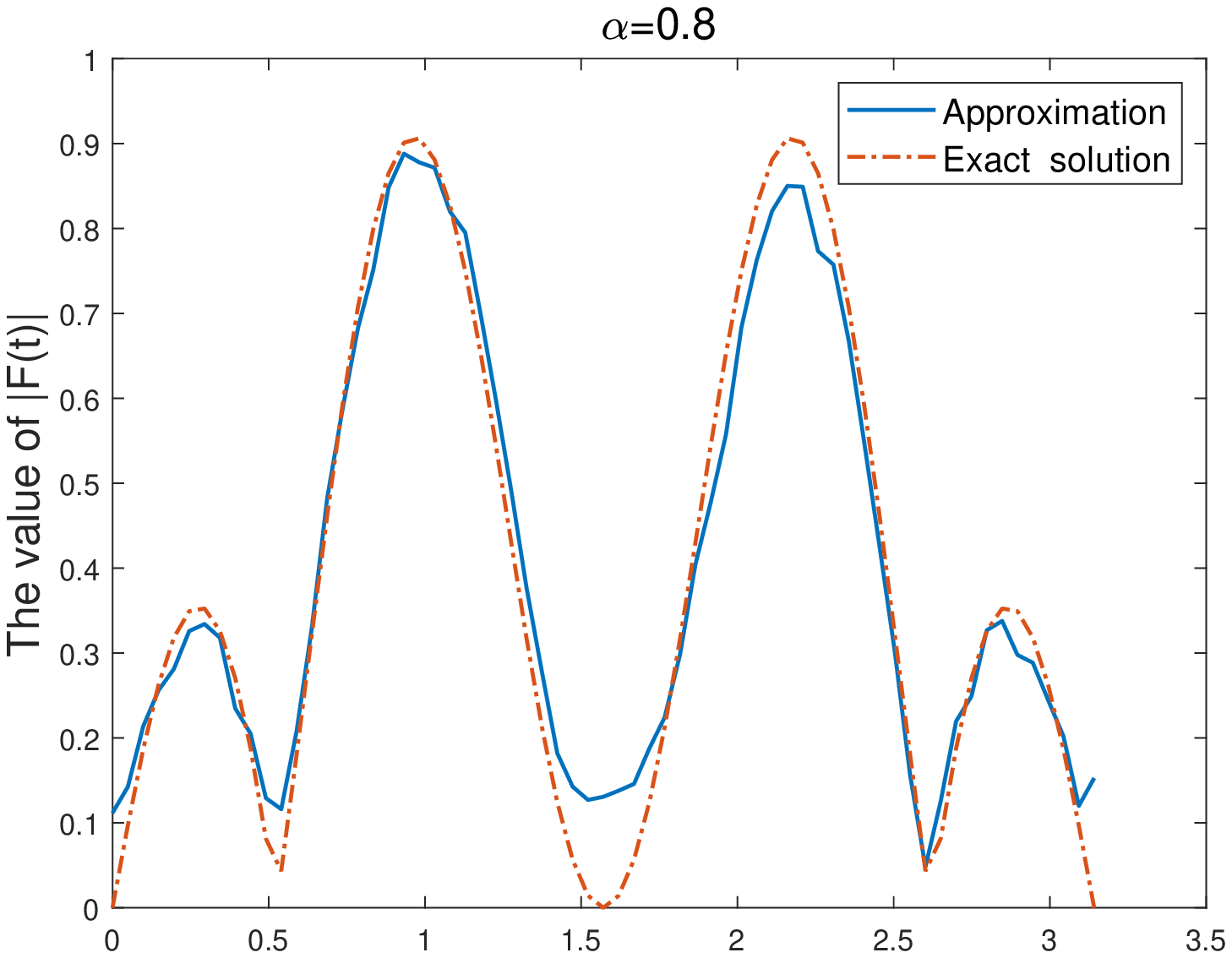}}\\
  \caption{Example 2: the exact solution is plotted against the reconstructed solution of $|F|$. (left) $\alpha =0.4$; (right) $\alpha=0.8$.}\label{num_soluton_f2}
\end{figure}

\section{Conclusion}\label{Conclusion}

In this paper, we have studied an inverse random source problem for the stochastic time fractional diffusion equation driven by a spatial Gaussian random field. By examining the equivalent two-point stochastic boundary value problem in the frequency domain, we show that the direct source problem is well-posed and the inverse source problem has a unique solution. The ill-posed nature is revealed for the inverse problem, i.e., the modulus of the Fourier coefficient of the source function decays in an $\alpha$ order of the frequency. The inverse problem is then converted into a phase retrieval problem which is to recover the original signal from its Fourier modulus and is implemented via the PhaseLift with random masks. The numerical results show that the method is effective to reconstruct the modulus of the source functions. 

The proposed approach based on the Fourier transform can be naturally extended to solve higher dimensional problems. For more complex cases, there are still several interesting problems to be investigated along this line of research. First, the source can be modeled by more general random processes such as spatial fractional Brownian motions. The analysis and results would remain the same for the direct problem. However, the uniqueness of the inverse problem may not be guaranteed due to the correlated kernel of fractional Brownian motions. Second, the Fourier transform requires more constraints about the functions $u$ and $f$, while the Laplace transform demands less constraints on these functions. Nevertheless, there is no efficient algorithm to deal with the phase retrieval problem on the Laplace transform. In addition, when it comes to the situation that $\alpha\in (1,2)$, more initial conditions are required to ensure the well-posedness of the model equation, which leads to the fact that the transform may not apply in this case. We hope to report the progress on these problems elsewhere in the future.

\end{document}